\documentclass{amsart}

\usepackage{amsfonts,amsmath,amsthm,amscd,amssymb}  

\newtheorem{theorem}{Theorem}[section]
\newtheorem{maintheorem}{Theorem}

\newtheorem{lemma}[theorem]{Lemma}

\newtheorem{maincorollary}{Corollary}

\newtheorem{definition}[theorem]{Definition}
\theoremstyle{definition}

\begin{document}

\title{On finite $ca$-$\mathfrak F$ groups and their applications}

\author{Evgeniy N. Myslovets, Alexander F. Vasil'ev}

\address{Francisk Scorina Gomel State University, Mathematical 
Department, Sovetskaya street, 104, Gomel, 246019, 
Belarus}

\email{myslovets@gmail.com}
\email{formation56@mail.ru}



\begin{abstract}
Let $\mathfrak F$ be a class of groups. A group $G$ is called
$ca$-$\mathfrak F$-group if its every non-abelian chief factor is simple
and $H/K \leftthreetimes C_G(H/K) \in \mathfrak F$ for every abelian chief
factor $H/K$ of $G$. In this paper, we investigate the structure of a finite
 $ca$-$\mathfrak F$-group. Properties of mutually permutable products of finite
$ca$-$\mathfrak F$-groups are studied.
\end{abstract}

\maketitle

\section{Introduction}

Only finite groups are considered. The concept of composition formation was
introduced by L.A.~Shemetkov~\cite{shem} and R.~Baer in an unpublished paper
(noted in~\cite[IV, p.\,370]{dh}). Every saturated formation is a composition 
formation. The class of all quasinilpotent groups is an example of composition,
but not saturated formation. Guo~Wenbin and
A.\,N.~Skiba~\cite{guo-skiba, guo-skiba-2} introduced the concept of 
qausi-$\mathfrak F$-group that is a generalization of quasinilpotency. 
In~\cite{guo-skiba} they proved that the class of all qausi-$\mathfrak F$-groups is
a composition formation if $\mathfrak F$ is a saturated formation containing all
nilpotent groups. In~\cite{guo-skiba, guo-skiba-2} some applications of 
formations of qausi-$\mathfrak F$-groups were considered.

In~\cite{ved} V.A.~Vedernikov introduced the definition of a $c$-supersoluble group.
Recall~\cite{ved} that a group $G$ is called $c$-supersoluble if every chief 
factor of $G$ is simple. In~\cite{vas-vas} A.F.~Vasil'ev and T.I.~Vasil'eva proved
that the class $\mathfrak U_{\mathrm{c}}$ of all $c$-supersoluble groups is a
composition but not a saturated formation. D.~Robinson (using notation: 
$SC$-group)~\cite{rob} established the structural properties of finite 
$c$-supersoluble groups.

In~\cite{mysl} the following generalization of $c$-supersolubility was 
proposed.

Let $\mathfrak F$ be a class of groups. Recall~\cite{shem-sk} that a chief 
factor $H/K$ of group $G$ is called {\it $\mathfrak F$-central} provided
$H/K \leftthreetimes G/C_G(H/K) \in \mathfrak F$.

\begin{definition}[\cite{mysl}]
\label{def:fca}
Let $\mathfrak F$ be a class of groups. A group $G$ is called a
$ca$-$\mathfrak F$-group if its every non-abelian chief factor is simple
and every abelian chief factor of $G$ is $\mathfrak F$-central. 
\end{definition}

The class of all $ca$-$\mathfrak F$-groups is denoted by $\mathfrak F_{ca}$.
If $\mathfrak F = \mathfrak U$ we have that $\mathfrak F_{ca} = \mathfrak U_c$.
 If $\mathfrak F = \mathfrak N \mathfrak A$, 
then $\mathfrak F_{ca} = (\mathfrak N \mathfrak A)_{ca}$ is the class 
of all groups whose every non-abelian chief factor is simple and
$Aut_G(H/K)$ is abelian for every abelian chief factor $H/K$. If
$\mathfrak F = \mathfrak G$ then $\mathfrak F_{ca}$ is the
class of all SNAC-groups~\cite{rob}, i.e the class of all groups whose all 
non-abelian factors are simple.

The class of all $ca$-$\mathfrak F$-groups is a composition 
formation~\cite{mysl}. Also in~\cite{mysl} some properties of the products of normal 
$ca$-$\mathfrak F$-subgroups were found.

Recall~\cite[\S 8]{shem-sk}, $Z^{\mathfrak F}_{\infty}(G)$ denotes
the $\mathfrak F$-hypercenter of a group $G$. $Z^{\mathfrak F}_{\infty}(G)$ is
 the product of all normal
subgroups $H$ of $G$ whose $G$-chief factors are $\mathfrak F$-central.

The following theorem is an extension of Robinson's result~\cite{rob} for
case when $\mathfrak F$ is a soluble saturated formation.

\begin{maintheorem}
\label{theorem:A}
Let $\mathfrak F$ be a soluble saturated formation. A 
group $G$ is a $ca$-$\mathfrak F$-group if and only if $G$ satisfies:

1. $G^{\mathfrak S} = G^{\mathfrak F}$;

2. If $G^{\mathfrak S} \neq 1$ then $G^{\mathfrak S} / Z(G^{\mathfrak S})$
is a direct product of $G$-invariant non-abelian simple groups; 

3. $Z(G^{\mathfrak S}) \subseteq Z^{\mathfrak F}_{\infty}(G)$
\end{maintheorem}

Following Carocca~\cite{car}, we say that $G = HK$ is the mutually
permutable product of subgroups $H$ and $K$ if $H$ permutes with every
subgroup of $K$ and $K$ permutes with every subgroup of $H$. The mutually
permutable products of supersoluble and $c$-supersoluble subgroups were 
investigated in many works of different authors (see~monograph~\cite{bal}).
A lot of papers were dedicated to the case where $G = HK$ is the mutually
permutable product of subgroups $H$ and $K$ which belong to a saturated
formation $\mathfrak F$. Therefore we have the following problem.

\textbf{Problem.} Lets $\mathfrak F$ be a composition formation. What is the 
structure of the group $G = HK$ where $H$ and $K$ are mutually permutable 
$\mathfrak F$-subgroups of $G$.

In this paper this problem is solving for a formation of $ca$-$\mathfrak F$-groups
where $\mathfrak F$ is a saturated formation contaning all supersoluble groups.

\begin{maintheorem}
\label{theorem:B}
Let $\mathfrak F$ be a saturated formation containing the
class $\mathfrak U$ of supersoluble groups. Let a group $G = HK$ be the product 
of the mutually permutable subgroups $H$ and $K$ of $G$. If $G$ is a 
$ca$-$\mathfrak F$-group then both $H$ and $K$ are $ca$-$\mathfrak F$-groups.
\end{maintheorem}

\begin{maincorollary}[\cite{bal2}]
Let the group $G = HK$ be the mutually permutable product of the subgroups $H$
and $K$ of $G$. If $G$ is a $c$-supersoluble group then both $H$ and $K$ are
$c$-supersoluble groups.
\end{maincorollary}

\begin{maincorollary}[\cite{beid-hein}]
Let the group $G = HK$ be the mutually permutable product of the subgroups $H$
and $K$ of $G$. If $G$ is a SNAC-group then both $H$ and $K$ are SNAC-groups.
\end{maincorollary}

\begin{maincorollary}
Let the group $G = HK$ be the mutually permutable product of the subgroups $H$
and $K$ of $G$. If $G$ is a $ca$-$\mathfrak N \mathfrak A$-group then both $H$ and $K$ 
are $ca$-$\mathfrak N \mathfrak A$-groups.
\end{maincorollary}

It is well known that in general, the product $G = HK$ of two normal 
supersoluble subgroups of a finite group $G$ need not be supersoluble. In 1957 
Baer~\cite{baer} established that such group $G$ will be a supersoluble 
if and only if the derived subgroup $G'$ of $G$ is nilpotent. The next theorem 
is an extension of this result.

\begin{maintheorem}
\label{theorem:C}
Let $\mathfrak F$ be a saturated formation containing the
class $\mathfrak U$ of supersoluble groups. Let the group $G = HK$ be the 
product of the mutually permutable $ca$-$\mathfrak F$-subgroups $H$ and $K$ of
$G$. If the derived subgroup $G'$ of $G$ is quasinilpotent, then $G$ is a
$ca$-$\mathfrak F$-group.
\end{maintheorem}

\setcounter{maincorollary}{0}

\begin{maincorollary}[\cite{bal2}]
Let the group $G = HK$ be the product of the mutually permutable
$c$-supersoluble subgroups $H$ and $K$ of $G$. If the derived subgroup $G'$ of
$G$ is quasinilpotent, then $G$ is $c$-supersoluble.
\end{maincorollary}
 
\begin{maincorollary}[\cite{as-sha}]
Let the group $G = HK$ be the product of the mutually permutable supersoluble
subgroups $H$ and $K$ of $G$. If the derived subgroup $G'$ of $G$ is nilpotent,
then $G$ is supersoluble.
\end{maincorollary}

\begin{maincorollary}
Let the group $G = HK$ be the product of the mutually permutable
$ca$-$\mathfrak N \mathfrak A$-subgroups $H$ and $K$ of $G$. If the derived
subgroup $G'$ of $G$ is quasinilpotent, then $G$ is
$ca$-$\mathfrak N \mathfrak A$-group.
\end{maincorollary} 

The following corollary extends~\cite{mysl} the properties of normal products of
$ca$-$\mathfrak F$-groups.

\begin{maincorollary}
Let $\mathfrak F$ be a saturated formation containing the class $\mathfrak U$ of
supersoluble groups. If $G = HK$ is the product of normal
$ca$-$\mathfrak F$-subgroups $H$ and $K$ of $G$ and the derived subgroup $G'$ of
$G$ is quasinilpotent, then $G$ is a $ca$-$\mathfrak F$-group.
\end{maincorollary}

\section{Preliminaries}

Standard notations, notions and results are used in the paper 
(see~\cite{dh,shem2}). Recall significant notions and notations for this paper.
$\mathbb{P}$ is the set of all prime numbers; 
$1$ is an identity group;
$H \leftthreetimes K$ is a semidirect product of groups $H$ and $K$; 
$\mathfrak{G}$ is the class of all groups;
$\mathfrak{S}$ is the class of all soluble groups; 
$\mathfrak{U}$ is the class of all supersoluble groups; 
$\mathfrak{U}_\mathrm c$ is the class of all $c$-supersoluble groups; 
$\mathfrak{N}$ is the class of nilpotent groups; 
$\mathfrak{N}_p$ is the class of all $p$-groups; 
$\mathfrak J$ is the class of all simple groups; 
$\mathfrak{A}(p-1)$ is the class of all abelian groups of exponent dividing $p-1$.

A {\it formation} is a homomorph $\mathfrak F$ of groups such that each group
$G$ has the smallest normal subgroup (called $\mathfrak F$-residual and denoted by
$G^\mathfrak F$) with quotient in $\mathfrak F$. A formation $\mathfrak F$ 
is said to be {\it saturated} if it contains each group $G$ with
$G / \Phi(G) \in \mathfrak F$. A formation $\mathfrak F$ is said to be
(normally) hereditary if it contains all (normal) subgroups of every group in 
$\mathfrak F$.

Let $\mathfrak F$ be a non-empty formation. $G_{\mathfrak F}$ denotes
$\mathfrak F$-radical of group $G$, i.e., the largest normal  
$\mathfrak F$-subgroup of $G$

A function $f: \mathbb{P} \rightarrow \{\textrm{formations of groups}\}$ is called a
{\it local formation function}. The symbol $LF(f)$ denotes the class of all groups such
that either $G = 1$ or $G \neq 1$ and $G/C_G(H/K) \in f(p)$ for every chief
factor $H/K$ of $G$ and every $p \in \pi (H/K)$. The class $LF(f)$ is a 
non-empty formation.

 For a formation $\mathfrak F$,
if there exists a formation function $f$ such that $\mathfrak F = LF(f)$ then
$\mathfrak F$ is called a {\it local formation}.  It is known that a formation 
$\mathfrak F$ is local if and only if it is saturated~\cite[IV, Theorem 4.6]{dh}.

A formation $\mathfrak F$ is said to be solubly saturated, composition, or Baer-local
formation if it contains each group $G$ with $G / \Phi(N) \in \mathfrak F$ for
 some soluble normal subgroup $N$ of $G$.
For every function $f$ of the form
$f : \mathfrak J \rightarrow \{\textrm{formations of groups}\}$ we put,
$CLF(f) = \{G \; \textrm{is a group} \; | \; G / C_G(H/K) \in f(A)$ for
every $A \in \mathcal K_{H/K})\}$. It is well known that a {\it composition
formation} (or a Baer-local formation if we use the terminology in~\cite{dh})
$\mathfrak F$ is exactly a class $\mathfrak F=CLF(f)$ for some function $f$ of
above-mention form. In this case, the function $f$ is said to be a
{\it composition satellite}~\cite{shem00} of the formation $\mathfrak F$.

A local function $f$ is called {\it an inner local function} if
$f(p) \subseteq LF(f)$ for every prime~$p$. Function $f$ is called {\it a 
maximal inner local function} of formation $\mathfrak F$ if $f$ is a maximal
element of set of all inner local functions of formation $\mathfrak F$.
Similarly, we can introduce the notion of the inner composition satellite and
maximal inner composition satellite.

Every local (composition) formation has the unique maximal inner local function
(composition satellite)~\cite[ch.~1]{shem2}.

We will use the following results.

\begin{lemma} [\cite{mysl}]
\label{lem:formation}
Let $\mathfrak F$ be a class of groups. Then the class $\mathfrak F_{ca}$ is a 
non-empty formation.
\end{lemma}

\begin{theorem} [\cite{mysl}]
\label{theorem:composition}
Let $\mathfrak F$ be a saturated formation and $f$ is its maximal inner local
function. Then the formation $\mathfrak F_{ca}$ is a composition formation and
has a maximal inner composition satellite $h$ such that
$h(N) = \mathfrak F_{ca}$, if $N$ is a non-abelian group and $h(N) = f(p)$, if 
$N$ is a simple $p$-group, where $p$ is a prime.
\end{theorem}

\begin{lemma}[\cite{vas-vas}]
\label{lem:lemma1}
Let $\mathfrak F$ be a formation and $N$ be a minimal normal subgroup of $G$
such as $|N| = p^{a}$ for some prime $p$. If $N$ contains in the subgroup $H$ of
$G$ and $H / C_{H}(U/V) \in \mathfrak F$ for every $H$-chief factor $U/V$ of
$N$, then $H / C_{H}(N) \in\mathfrak {N}_{p}\mathfrak F$.
\end{lemma}

\begin{lemma} [\cite{bal}]
\label{lem:factor}
Assume that the subgroups $A$ and $B$ of the group $G$ are mutually
permutable and that $N$ is a normal subgroup of $G$. Then the subgroups $AN / N$
and $BN / N$ are mutually permutable in $G/N$. 
\end{lemma}

\begin{lemma}[\cite{bal}]
\label{lem:lemma_4.3.3}
Let the group $G = AB$ be the mutually permutable product of the subgroups $A$
and $B$. Then: 

1. If $N$ is a maximal normal subgroup of $G$, then 
$\{AN, BN, (A \cap B)N\} \subseteq \{N, G\}$. 

2. If $N$ is a non-abelian minimal normal subgroup of $G$, then
$\{A \cap N, B \cap N\} \subseteq \{N, 1\}$ and $N = (N \cap A)(N \cap B)$ (that
is, $N$ is prefactorised with respect to $G = AB$). 

3. If $N$ is a minimal normal subgroup of $G$, then $N \leq A \cap B$ or
$[N, A \cap B] = 1$.

4. If $N$ is a minimal normal subgroup of $G$, then
$\{A \cap N, B \cap N \} \subseteq \{N, 1 \}$.

5. If $N$ is a minimal normal subgroup of $G$ contained in $A$ and
$B \cap N = 1$, then $N \leq C_G(A)$ or $N \leq C_G(B)$. If furthermore $N$ is
not cyclic, then $N \leq C_G(B)$. 
\end{lemma}

\begin{lemma} [\cite{bal}]
\label{lem:core}
Let the group $G = AB$ be the product of the mutually permutable subgroups $A$ 
and $B$ and let $\mathfrak F$ be a saturated formation containing the class 
$\mathfrak U$ of all supersoluble groups. If $(A \cap B)_G = 1$, then
$G \in \mathfrak F$ if and only if $A\in \mathfrak F$ and $B \in \mathfrak F$. 
\end{lemma}

\section{Proof of theorem~\ref{theorem:A}}

In this section we prove the theorem that describes the structure of finite
$ca$-$\mathfrak F$-group.

\begin{lemma}.
\label{lem:bal_1_6_4}
Let $\mathfrak F$ be a soluble formation containing the class $\mathfrak U$ of 
all supersoluble groups. If $G$ is a $ca$-$\mathfrak F$-group then the following 
statements hold:

1. $G^\mathfrak S \leq C_G(G_\mathfrak S)$;

2. $(G^\mathfrak S)_\mathfrak S \leq Z(G^\mathfrak S)$.
\end{lemma}

\begin{proof}
Prove the statement~1. Obviously that all chief factors of group $G$ below
subgroup $G_\mathfrak S$ are $\mathfrak F$-central. Hence subgroup
$G_\mathfrak S$ is $\mathfrak F$-hypercentral and thus it is subgroup of 
$\mathfrak F$-hypercenter $Z^{\mathfrak F}_\infty(G)$. By 
Corollary~9.3.2~\cite{shem2} we have that
$G^{\mathfrak F} \leq C_G(Z^{\mathfrak F}_\infty(G))$. Since $\mathfrak F$ is a 
soluble formation, $G^{\mathfrak S}\leq G^{\mathfrak F} \leq C_G(G_\mathfrak S)$.

Prove the statement~2. Let $R = (G^\mathfrak S)_\mathfrak S$. Since
$R$ char $G^\mathfrak S \unlhd G$, it follows $R \unlhd G$. Therefore
$R \leq G_\mathfrak S$. Hence $G^\mathfrak S \leq C_G(R)$ by statement 1 of the 
Lemma. The statement 2 is true. 
\end{proof}
 
\begin{proof}[Proof of theorem~\ref{theorem:A}]
Denote by $D$ the soluble residual $G^\mathfrak S$ of group $G$.

Let $G$ be a $ca$-$\mathfrak F$-group. If $G$ is soluble, then $D = 1$
and $G \in \mathfrak F$. So $G$ satisfies the Statements 1, 2, and 3.
We assume that group $G$ is not soluble. Then $D \neq 1$. 

Since $\mathfrak F$ is a soluble formation,
$G/G^\mathfrak F \in \mathfrak F \subseteq \mathfrak S$. Hence
$D \subseteq G^\mathfrak F$. Since $\mathfrak F_{ca}$ is a formation, it follows
$G / D \in \mathfrak F_{ca}$. By solvability of quotient $G / D$ we have that
$G / D \in \mathfrak F$. Hence $G^\mathfrak F \subseteq D$ and
$D = G^\mathfrak F$. The Statement 1 holds. Note that all chief factors of $G$ 
below $Z(D)$ are abelian and therefore are $\mathfrak F$-central. This means, 
that $Z(D)$ is $\mathfrak F$-hypercentral and the Statement 3 holds.

We show that $D / Z(D)$ is a direct product of $G$-invariant simple groups.

Assume that $Z(D) = 1$. Let $N_1$ be a minimal normal subgroup of $G$ contained 
in $D$. If $N_1$ is abelian, then it follows from $N_1 \leq D_\mathfrak S$ and 
the Statement 2 of Lemma~\ref{lem:bal_1_6_4} that $N_1 \leq Z(D) = 1$. Hence 
$N_1$ is non-abelian. Since $G \in \mathfrak F_{ca}$, we have that $N_1$ is a
simple. 
Note that $G / C_G(N_1)$ is isomorphic to a subgroup of 
$\mathrm{Aut}(N_1)$ and 
$N_1C_G(N_1)/C_G(N_1)$ is isomorphic to $\mathrm{Inn}(N_1)$. So
$G / N_1C_G(N_1) \simeq (G / C_G(N_1)) / (N_1C_G(N_1) / C_G(N_1))$ is isomorphic 
to a subgroup of $\mathrm{Aut}(N_1) / \mathrm{Inn}(N_1)$. From the validity of 
the Schreier conjecture, it follows that $G / N_1C_G(N_1)$ is soluble. 
Then $D \leq N_1C_G(N_1)$. 
Hence $D = D\cap N_1C_G(N_1) = N_1(D\cap C_G(N_1)) = N_1C_D(N_1)$ and
$N_1 \cap C_D(N_1) = 1$. If $D = N_1$, then the Statement 2 holds. Assume that
$D$ is not simple. Therefore, $C_D(N_1) \neq 1$. The Statement 2 holds in the
case when $C_D(N_1)$ is simple. Assume that $C_D(N_1)$ is not a simple
and let $N_2$ be a minimal normal subgroup of $G$ contained in $C_D(N_1)$. Since 
$Z(D) = 1$ and the Statement 2 of lemma~\ref{lem:bal_1_6_4}, it follows that 
$N_2$ is a simple non-abelian subgroup. By the above $D = N_2C_D(N_2)$. 
By Dedekind identity $C_D(N_1) = C_D(N_1) \cap N_2C_D(N_2) = $
$N_2(C_D(N_1) \cap C_D(N_2)) = N_2C_L(N_2)$, where $L = C_D(N_1) \cap C_D(N_2)$. 
Then $D = N_1N_2C_L(N_2)$. Applying above to $C_D(N_2)$ and etc.
we can conclude that $D = N_1 \times N_2 \times \cdots \times N_t$ is the direct
product of minimal normal subgroups of $G$, each of them simple, as desired. 
So the Statement 2 holds.

Let $Z(D) \neq 1$. Since $G/Z(D) \in \mathfrak F_{ca}$ and
$(G/Z(D))^\mathfrak S = D/Z(D)$, the Statement 1 and 3 holds for $G/Z(D)$.
Denote $T / Z(D) = Z(D / Z(D))$. Then $T$ is a normal soluble subgroup of $D$.
By lemma~\ref{lem:bal_1_6_4} $T$ is contained in the center $Z(D)$.
Therefore $Z(D / Z(D)) = 1$. By the above the Statement 1 holds for $G/Z(D)$. 

Conversely, assume that a group $G$ satisfies the Statements 1, 2 and 3.
We consider a chief series of $G$ which passes through the subgroup $D=G^{\mathfrak S}$.
 Note the all chief factors 
above $D$ are abelian and $\mathfrak F$-central. By the Statement~2 the quotient
$D / Z(D)$ is the direct product of minimal normal subgroups of $G / Z(D)$, 
which are simple. All chief factors of $G$ below $Z(D)$ are
$\mathfrak F$-central by the Statement 3. By virtue of Jordan-Holder's theorem 
for groups with operators~\cite[A, 3.2]{dh} and the Definition~\ref{def:fca},
$G \in \mathfrak F_{ca}$.
\end{proof}

\section{Proof of theorem~\ref{theorem:B} and~\ref{theorem:C}}

In this section we prove some properties of the mutually permutable products of
$ca$-$\mathfrak F$-groups.

\begin{proof}[Proof of Theorem~\ref{theorem:B}]
Assume that that this theorem is false and let $G$ be a counterexample of minimal 
order. Let $N$ be a minimal normal subgroup of $G$. If $N = G$, then $G$ is 
simple. Hence $H \in \mathfrak F_{ca}$ and
$K \in \mathfrak F_{ca}$. Assume $N \neq G$. By Lemma~\ref{lem:factor}
$G/N = HN/N \cdot KN/N$ is the mutually permutable product of subgroups $HN/N$
and $KN/N$ of $G/N$. Note that $G/N \in \mathfrak F_{ca}$. Then all 
conditions of the theorem hold for $G/N$. Therefore
$HN/N \simeq H/(H \cap N) \in \mathfrak F_{ca}$ and 
$KN/N \simeq K/(K \cap N) \in \mathfrak F_{ca}$. Since
$\mathfrak F_{ca}$ is a formation by Lemma~\ref{lem:formation}, it
follows that $N$ is the unique minimal normal subgroup of $G$.

Let $N$ be a non-abelian group. Then $N$ is simple. According to 
Lemma~\ref{lem:lemma_4.3.3} we should consider the following cases.

1. Let $H \cap N = K \cap N = N$. Then $N \leq H \cap K$,
$H/(H \cap N) = H/N \in \mathfrak F_{ca}$ and
$K/(K \cap N) = K/N \in \mathfrak F_{ca}$. Hence $H$ and $K$ are 
$ca$-$\mathfrak F$-groups, a contradiction.

2. Let $H \cap N = K \cap N = 1$. Then $H/(H \cap N) \simeq H$ and
$K/(K \cap N) \simeq K$ are $ca$-$\mathfrak F$-groups, a contradiction.

3. Let $H \cap N = N$ and $K \cap N = 1$. Then
$H/(H \cap N) = H/N \in \mathfrak F_{ca}$ and $H$ is a
$ca$-$\mathfrak F$-group and $K/(K \cap N) \simeq K$ is a
$ca$-$\mathfrak F$-group. A contradiction.

4. Let $H \cap N = 1$ and $K \cap N = N$. This case is considered similarly to 
the case~3. 
\end{proof}

 To prove the Theorem C we need the following results.

\begin{lemma}
\label{lem:minimal}
Let the group $G$ has the unique minimal normal subgroup
$N = N_1 \times \dots \times N_t$ and $N_i$ are isomorphic simple non-abelian
groups for all $i = 1, \dots, t$. If $N \subseteq H$, where $H$ is a
$ca$-$\mathfrak F$-subgroup of $G$, then $N_i \lhd H$ for all $i = 1, \dots, t$.
\end{lemma}
\begin{proof}
Let $i \in\{1, \dots, t\}$. Consider normal closure
$N_i^H= \langle {N_i}^x | x\in H\rangle$ of subgroup $N_i$ in $H$. Note that 
$N_i \lhd \lhd \ G$. Hence $N_i \lhd \lhd \ H$. By the Lemma~9.17~\cite{isaacs} 
we have that $N_i^H$ is a minimal normal subgroup of $H$. Since subgroup $N_i^H$
is non-abelian and isomorphic to the chief factor of $ca$-$\mathfrak F$-subgroup
$H$, then $N_i^H$ is simple. Then, by $N_i \lhd \lhd \ N_i^H$ we have
that $N_i^H = N_i$. Hence $N_i \lhd H$ for all $i = 1, \dots, t$.
\end{proof}

\begin{lemma}
\label{lem:equal}
Let $\mathfrak F$ be a composition formation and $f$ is an inner composition
satellite of $\mathfrak F$. Let a group $G$ has the unique minimal normal 
subgroup $N$ and $N$ is an abelian $p$-group for some prime $p$. The chief factor $N$ of 
$G$ is $\mathfrak F$-central in $G$ if and only if $G / C_G(N) \in f(p)$.
\end{lemma}

\begin{proof}
Let $G / C_{G}(N) \in f(p)$. Consider semidirect product
$R = N \leftthreetimes G/C_G(N)$. Note that $N$ is the unique minimal normal
subgroup of $R$ and $C_R(N) = N$. Then
$R/C_R(N) \simeq G/C_G(N) \in f(p) \subseteq \mathfrak F$. Hence
$R \in \mathfrak F$, i.e. the chief factor $N/1$ of $G$ is $\mathfrak F$-central.

Conversely, assume that $N$ is $\mathfrak F$-central chief factor of $G$.
Then $R = N \leftthreetimes G/C_{G}(N) \in \mathfrak F$, where $N$ is the
unique minimal normal subgroup of $R$ and $C_G(N) = N$. Hence
$R/C_R(N) \simeq G/C_G(N) \in f(p)$.
\end{proof}

\begin{lemma}
\label{lem:abelian}
Let $\mathfrak F$ be a formation and
$\mathfrak A(p-1) \subseteq \mathfrak F$. Let $G=HK$ be the mutually permutable 
products of subgroup $H$ and $K$, where $H, K \in \mathfrak N_p \mathfrak F$ and 
$G \in \mathfrak N_p \mathfrak A$. Then $G \in \mathfrak N_p \mathfrak F$.
\end{lemma}

\begin{proof}
Assume that this lemma is false and let $G$ be a counterexample of minimal order. 
Let $N$ be a minimal normal subgroup of $G$. We can assume without loss
of generality that $G \neq N$. By Lemma~\ref{lem:factor}
$G/N = HN/N \cdot KN/N$ is the mutually permutable product of subgroups $HN/N$
and $KN/N$ of $G/N$. Note that $HN/N \in \mathfrak N_p \mathfrak F$,
 $HN/N \in \mathfrak N_p \mathfrak F$ and
$G/N \in \mathfrak N_p \mathfrak A$. Then all conditions of 

the Lemma 4.3 hold for
$G/N$. Therefore $G/N \in \mathfrak N_p \mathfrak F$. Since
$ \mathfrak N_p \mathfrak F$ is a formation, it follows that $N$ is the unique 
minimal normal subgroup of $G$. We note that $N$ is a $q$-group for some prime
$q \neq p$. Since $G \in \mathfrak N_p \mathfrak A$ and $O_p(G) = 1$, we have 
that $G \in \mathfrak A$. Therefore $H \in \mathfrak F$ and $K \in \mathfrak F$.
Since $N$ is the unique minimal normal subgroup of $G$, it follows that $G$ is a
cyclic $q$-group. Since $G = HK$, we have that $G = H$ or $G = K$, 
i.e. $G \in \mathfrak F$.
\end{proof}

\begin{proof}[Proof of Theorem~\ref{theorem:C}]
Assume that this theorem is false and let $G$ be a counterexample of minimal 
order. Let $N$ be a minimal normal subgroup of $G$. If $N = G$, then $G$ is
simple. Hence $G \in \mathfrak F_{ca}$. Assume $N \neq G$. By 
Lemma~\ref{lem:factor} $G/N = HN/N \cdot KN/N$ is the mutually permutable 
product of subgroups $HN/N$ and $KN/N$ of $G/N$. Note that the derived subgroup 
$(G/N)'$ of $G/N$ is quasinilpotent. Then all conditions of the Theorem hold for
$G/N$. Therefore $G/N \in \mathfrak F_{\mathrm {ca}}$. Since
$\mathfrak F_{ca}$ is a formation by Lemma~\ref{lem:formation}, it
follows that $N$ is the unique minimal normal subgroup of $G$.

Let $N$ be a non-abelian group. Then $N = N_1 \times \dots \times N_t$, where
$N_i$ are isomorphic simple non-abelian groups for all $i = 1, \dots, t$.
According to Lemma~\ref{lem:lemma_4.3.3} we should consider the following cases.

1. Let $H \cap N = K \cap N = N$. Then $N \subseteq H \cap K$. Since $H$ and
$K$ are $ca$-$\mathfrak F$-subgroups and $N = N_1 \times \cdots \times N_t$, it 
follows that $N_i \lhd H$ and $N_i \lhd K$ by Lemma~\ref{lem:minimal}. Hence
by $G = HK$ we have that $N_i \lhd G$ for all $i = 1, \dots, t$. Since $N$ is
the unique minimal normal subgroup of $G$, it follows that $t = 1$
and $N$ is simple. Since $G/N \in \mathfrak F_{ca}$, we have that
$G \in \mathfrak F_{ca}$. A contradiction.

2. Let $H \cap N = K \cap N = 1$. Then $N = (H \cap N)(K \cap N) = 1$ by 
Lemma~\ref{lem:lemma_4.3.3}(2). A contradiction with choice of $N$.

3. Let $H \cap N = N$ and $K \cap N = 1$. Then $N \subseteq H$ and
$N \leq C_G(H)$ or $N \leq C_G(K)$ by Lemma~\ref{lem:lemma_4.3.3}(5). Since
$N \leq H$ and $N$ is non-abelian, we have that $N \leq C_G(K)$. Since
$N = N_1 \times \dots \times N_t$, it follows that $N_i \leq C_G(K)$ for all
$i = 1, \dots, t$. By $N_i \leq H$ we have that $N_i \lhd H$ for all
$i = 1, \dots, t$ by Lemma~\ref{lem:minimal}. By $G = HK$ we have that
$N_i \lhd G$. Since $N$ is the unique minimal normal subgroup of $G$, it 
follows that $N = N_i$ and $N$ is simple. Since
$G/N \in \mathfrak F_{ca}$, we have that
$G \in \mathfrak F_{ca}$. A contradiction.

4. Let $H \cap N = 1$ and $K \cap N = N$. This case is considered similarly to 
the case~3. 

Assume $N$ is an abelian group. Then $N$ is a $p$-group for some prime $p$.
By Theorem~\ref{theorem:composition} formation $\mathfrak F_{ca}$ has the
maximal inner composition satellite $h$ such that $h(N) = f(p)$, where $f$ is a
maximal inner local function of $\mathfrak F$. According to
Lemma~\ref{lem:lemma_4.3.3} we should consider the following cases.

1. Let $H \cap N = K \cap N = N$. Then $N \subseteq H \cap K$. Let $U / V$ is 
any $H$-chief factor of $N$. Since $H \in \mathfrak F_{ca}$, it follows that
$H / C_H(U/V) \in h(p)$. By Lemma~\ref{lem:lemma1} we have that
$H / C_H(N) \in \mathfrak N_ph(p)= h(p)$. Similarly we can show that
$K / C_K(N) \in \mathfrak N_ph(p) = h(p)$. Note the group
$G / C_G(N) = HC_G(N) / C_G(N) \cdot KC_G(N)/ C_G(N)$ is the mutually permutable
product of subgroups $HC_G(N) / C_G(N)$ and $KC_G(N)/ C_G(N)$ of $G / C_G(N)$.

Since $N \subseteq G'$ and $G'$ is quasinilpotent, it follows that
$G'/C_{G'}(N) \in \mathfrak N_p$ by Lemma~\ref{lem:lemma1}. So
$(G / C_G(N))' = G'C_G(N)/C_G(N) \simeq G'/C_{G'}(N) $ is a $p$-group.
Since $G / C_G(N) / (G / C_G(N))' \in \mathfrak A$, it follows that
$G / C_G(N) \in \mathfrak N_p\mathfrak A$. By Lemma~\ref{lem:abelian} for
$G / C_G(N)$ we have that
$G / C_G(N) \in \mathfrak N_ph(p)= h(p) \subseteq \mathfrak F_{ca}$.
Therefore $G \in \mathfrak F_{ca}$. A contradiction.

2. Let $H \cap N = K \cap N = 1$. Then $N \notin H \cap K$ and
$(H \cap K)_G = 1$. If $H^\mathfrak S = 1$ and $K^\mathfrak S = 1$, then $H$ and
$K$ are soluble. Hence $H \in \mathfrak F$ and $K \in \mathfrak F$. By 
Lemma~\ref{lem:core} we have that
$G \in \mathfrak F \subseteq \mathfrak F_{ca}$, a contradiction. 
Hence we can assume without loss of generality that $H^\mathfrak S \neq 1$. 
Then $H^\mathfrak S \trianglelefteq G$ by Corollary~4.3.6~\cite{bal}. Hence
$N \leq H^\mathfrak S$. Therefore $N \leq H \cap N = 1$, a contradiction.

3. Let $H \cap N = N$ and $K \cap N = 1$. Assume that $N$ is non-cyclic subgroup.
Then $N \leq C_G(K)$ by Lemma~\ref{lem:lemma_4.3.3}(5). Hence
$K \subseteq C_G(N)$ and $G/C_G(N) = HC_G(N)/C_G(N) \cdot KC_G(N) / C_G(N) =$ 
$HC_G(N)/C_G(N) \simeq H/(C_G(N) \cap H) = H/C_H(N)$. Since $N \subseteq H$ and
$H \in \mathfrak F_{ca}$, it follows that $H/C_H(N) \in h(p)$ by 
Lemma~\ref{lem:lemma1}. By Lemma~\ref{lem:equal} we have that factor $N$ is 
$\mathfrak F$-central chief factor of $G$. Then
$G \in \mathfrak F_{ca}$, a contradiction. Let $N$ be a cyclic group. 
Then $|N| = p$ and $G / C_G(N)$ is a cyclic group of order dividing $p - 1$. 
Hence $G / C_G(N) \in \mathfrak A(p-1) \subseteq f(p) = h(p)$. Since
$G/N \in \mathfrak F_{ca}$, it follows that $G \in \mathfrak F_{ca}$, a 
contradiction.

4. Let $H \cap N = 1$ and $K \cap N = N$. This case is considered similarly to 
the case~3.                             
\end{proof}

\section{Final remarks}

Many different specific examples of composition formations containing all
supersoluble groups can be built using the concept of $ca$-$\mathfrak F$-group.
 
According to~\cite{dh}, a {\it rank function} is a map
$R \; : \; p \rightarrow R(p)$ which associates with each prime $p$ a set $R(p)$ 
of natural numbers. With each rank function $R$ we associate a class~\cite{dh}

\centerline{$\mathfrak F(R) = (G \in \mathfrak S \; | \;$ for all prime
$p \in \mathbb P$ each $p$-factor of $G$}
\centerline{has rank in $R(p))$,}

that is a formation.

If $\mathfrak F(R)$ is a saturated formation, then rank function is called a
{\it saturated} (see~\cite[p.~484]{dh}). A rank function $R$ is said to have 
{\it full characteristic} if $R(p) \neq \emptyset$ for all $p \in \mathbb P$.

Note that if $R$ is a saturated rank function of full characteristic,
by~\cite[IV, 4.3]{dh}, we have $1 \in R(p)$ for all prime $p \in \mathbb P$ and
therefore $\mathfrak U \subseteq \mathfrak F(R)$.

If a rank function $R$ is defined, then for all prime $p \in \mathbb P$
are defined~\cite{dh}

\centerline{$\pi(G) = R(p) \cap \mathbb P$ and}
\centerline{$e(p) = \{p^m - 1 \; | \; m \in R(p)\}$.}

By $\mathfrak A_{\pi(p)'}(e(p))$ we denote a class of abelian $\pi(p)'$-group
with exponent dividing $e(p)$ that is a formation.

According to~\cite{dh} the following lemma holds.

\begin{lemma}[\cite{dh}]
Let $R$ is a saturated rank function of full characteristic. Then $R$ satisfies
the following conditions

{\bf RF1:} If $n \in R(p)$ and $m \; | \; n$, then $m \in R(p)$;

{\bf RF2:} If $\{m, n\} \in R(p)$, then $mn \in R(p)$;

{\bf RF3:} If $p$ and $q$ are distinct primes with $q \in R(p)$ and if
$m \in R(p)$, then $q^m - 1 \in R(p)$;

{\bf RF4:} If $p, q \in \mathbb P$ and $r \in \mathbb N$ satisfy
the following conditions:

(i) $p \; | \; (q^m - 1) $ for some $m \in R(p)$,

(ii) $q \; | \; (p^n - 1) $ for some $n \in R(p)$,

(iii) $r \; | \; (p^k - 1) $ for some $k \in R(p)$,

(iv) $p \in R(p)$, $r \in R(p)$,

then $r \in R(p)$.
\end{lemma}

Local function of formation $\mathfrak F(R)$ in the case when $R$ is a saturated 
rank function is described in theorem~2.18~\cite[p.~490]{dh} which we form as 
lemma.

\begin{lemma}
Let $R$ is a rank function and let $\hat{\mathfrak F}(R)$ is a local formation
defined by local function $f$ such that
$f(p) = \mathfrak A_{\pi(p)'}(e(p))\mathfrak S_\pi(p)$ for all prime $p$.
Then any two of the following statements are equivalent:

(a) $R$ is a saturated rank function;

(b) $R$ satisfies Conditions {\bf RF1-RF4};

(c) $\hat{\mathfrak F} = \mathfrak F$.
\end{lemma}

\setcounter{maintheorem}{2}
\setcounter{maincorollary}{3}

\begin{maincorollary}
Let $R$ be a saturated rank function of full characteristic and the group
$G = HK$ be the mutually permutable product of the subgroups $H$ and $K$ of $G$. 
If $G$ is a $ca$-$\mathfrak F(R)$-group, then $H$ and $K$ are also
$ca$-$\mathfrak F(R)$-groups.
\end{maincorollary}

\setcounter{maintheorem}{3}
\setcounter{maincorollary}{4}

\begin{maincorollary}
Let $R$ be a saturated rank function of full characteristic and the group 
$G = HK$ be the mutually permutable product of the
$ca$-$\mathfrak F(R)$-subgroups $H$ and $K$ of $G$. If the derived subgroup $G'$
of $G$ is quasinilpotent, then $G$ is a $ca$-$\mathfrak F(R)$-group.
\end{maincorollary}


\begin{thebibliography}{99}

\bibitem{as-sha}
M. Asaad, A. Shaalan,
\emph{On the supersolvability of finite groups},
Arch. Math. (Basel),
N.\textbf{53}, 1989, pp.318--326.

\bibitem{baer}
R. Baer,
\emph{Classes of finite groups  and  their  properties},
Illinois J.  Math.
N.\textbf{1}, 1957, pp.115--187.

\bibitem{bal}
A. Ballester-Bolinches, R. Esteban-Romero, M. Asaad,
\emph{Products of Finite Groups},
Walter de Gruyter, 2010.

\bibitem{bal2}
A. Ballester-Bolinches, J. Cossey, M.C. Pedraza-Aguilera,
\emph{On mutually permutable products of finite groups},
Journal of Algebra,
N.\textbf{294}, 2005, pp.127--135.

\bibitem{beid-hein}
J.C. Beidleman, H. Heineken,
\emph{Group classes and mutually permutable products},
Journal of Algebra,
N.\textbf{297}, 2006, pp.409--416.

\bibitem{car}
A. Carocca,
\emph{p-supersolvability of factorized finite groups},
Hokkaido Math. J,
N.\textbf{21}, 1992, pp.395--403.

\bibitem{dh}
K. Doerk, T. Hawkes,
\emph{Finite soluble groups},
Walter de Gruyter, 1992.

\bibitem{gasch} 
W. Gasch\"{u}tz,
\emph{Zur Theorie der endlichen aufl\"{o}sbaren Gruppen}, 
Math. Z.,
80, N.\textbf{4}, 1963,
S.~300--305.

\bibitem{guo-skiba}
W. Guo, A.N. Skiba,
\emph{On finite quazi-$\mathfrak F$-groups},
Communication in Algebra,
N.\textbf{37} 2009, pp.~470--481.

\bibitem{guo-skiba-2}
W. Guo, A.N. Skiba,
\emph{On some classes of finite quazi-$\mathfrak F$-groups},
Journal of Group Theory,
N.\textbf{12}, 2009, pp.~407--417.

\bibitem{isaacs}
I.M. Isaacs,
\emph{Finite Group Theory},
Graduate Studies in Mathematics,
N.\textbf{92}, 2008.

\bibitem{mysl}
E.N. Myslovets,
\emph{On finite $ca$-$\mathfrak F$-groups},
Problems of Physics, Mathematics and Technics,
2, N.\textbf{19}, 2014, pp.~64--68.

\bibitem{rob}
D.J.S. Robinson,
\emph{The structure of finite groups in which permutability is a transitive 
relation},
J. Austral. Math. Soc.,
N.\textbf{70}, 2001, pp.143--149.

\bibitem{shem00}
A.N. Skiba, L.A. Shemetkov,
\emph{Multiply $w$-composition formations of finite groups},
Ukrainsk. Math. Zh,
52, N.\textbf{6}, 2000, pp.783--797.

\bibitem{shem}
L.\,A.~Shemetkov,
\emph{Two directions in the development of the theory of non-simple finite 
groups},
Russian Mathematical Surveys,
30, N.~\textbf{2}, 1975, pp.~185--206.

\bibitem{shem2}
L.N. Shemetkov,
\emph{Formations of finite groups},
Nauka, 1978.

\bibitem{shem-sk}
L.A. Shemetkov, A.N. Skiba,
\emph{Formation of algebraic systems},
Nauka, 1989.

\bibitem{vas-vas}
A.F. Vasil'ev, T.I. Vasil'eva,
\emph{On finite groups with simple chief factors},
Izv. Vuzov. Ser. Mathematics,
426, N.\textbf{11}, 1997, pp.10--14.

\bibitem{ved}
V.A. Vedernikov,
\emph{On some classes of finite groups},
Dokl. Akad. Nauk BSSB,
2, N.\textbf{10}, 1988, pp.872--875.

\end{thebibliography}
\end{document}